\documentclass[english]{article}
\pdfoutput=1
\usepackage[T1]{fontenc}
\usepackage[letterpaper]{geometry}
\usepackage{float}
\usepackage{amsmath}
\usepackage{graphicx}
\usepackage{color}
\usepackage{epstopdf}

\usepackage[colorlinks, hypertexnames=false]{hyperref}
\hypersetup{linkcolor=blue, citecolor=red} 
\usepackage{algorithm}
\usepackage{algorithmic}

\usepackage{adjustbox}
\usepackage{bm}
\usepackage{multirow}
\usepackage{array}
\usepackage{mathrsfs}
\usepackage{amsmath, amssymb, amsthm}
\usepackage{subfigure}
\usepackage{tikz}
\numberwithin{equation}{section}
\usepackage{pifont}
\newcommand*{\Max}{\mathop{\mathrm{max}}}
\newcommand*{\Min}{\mathop{\mathrm{min}}}
\graphicspath{{./Figures/}}

\newtheorem{definition}{Definition}[section]
\newtheorem{theorem}{Theorem}[section]

\newtheorem{remark}{Remark}[section]

\usepackage{amsfonts} 
\usepackage{color}
\usepackage{adjustbox}
\definecolor{black}{rgb}{0,0,0}

\definecolor{red}{rgb}{1,0,0}

\definecolor{blue}{rgb}{0,0,1}

\usepackage{multirow}

\makeatother
\usepackage{babel}
\usepackage{authblk}
\title{}

\title{\textbf{Multiscale finite element method for Stokes-Darcy model}}
\author{Yachen Hong\thanks{Department of Mathematics,East China Normal
		University, Shanghai. }, \quad
	Wenhan Zhang \thanks{Department of Mathematics,East China Normal
		University, Shanghai. }, \quad
	Lina Zhao\thanks{Department of Mathematics, City University of Hong Kong, Kowloon Tong, Hong Kong SAR.}\quad
Haibiao Zheng\thanks{Department of Mathematics, Key Laboratory of MEA(Ministry of Education) and Shanghai Key Laboratory of PMMP, East China Normal University, Shanghai.}\;
	
}

\begin{document}
\maketitle
	\begin{abstract}

This paper explores the application of the multiscale finite element method (MsFEM) to address steady-state Stokes-Darcy problems with BJS interface conditions in highly heterogeneous porous media. We assume the existence of multiscale features in the Darcy region and propose an algorithm for the multiscale Stokes-Darcy model. During the offline phase, we employ MsFEM to construct permeability-dependent offline bases for efficient coarse-grid simulation, with this process conducted in parallel to enhance its efficiency. In the online phase, we use the Robin-Robin algorithm to derive the model's solution. Subsequently, we conduct error analysis based on $L^2$ and $H^1$ norms, assuming certain periodic coefficients in the Darcy region.
To validate our approach, we present extensive numerical tests on highly heterogeneous media, illustrating the results of the error analysis.
		
	\end{abstract}
	
	
	\pagestyle{myheadings} \thispagestyle{plain} \markboth{Lina}
	{MsFEM for heterogeneous Stokes-Darcy model}

	\section{Introduction}\label{sec:offline1}
	    In this paper, we design and analyze an efficient numerical method based on the  multiscale finite element
	    method (MsFEM) framework for Stokes-Darcy system. The Stokes-Darcy problem has gained significant attention over the last decade, particularly following the influential works by Discacciati et al.\cite{discacciati2002mathematical} and Layton et al \cite{layton2002coupling}. Such models can be used to describe pysiological phenomena like hydrological systems in which surface water percolates through rocks and sand, and various industrial processes involving filtration, reservoir simulation, nuclear water storage and underground water contamination. In these fields, it is often encountered with situations involving multiple scales; for example complex rock matrices when modelling sub-surface flows or random placements of buildings, people and trees in the context of urban canopy flows. Hence, it is necessary to take into consideration the presence of multiple scales in the Stokes-Darcy model.
	
	These multiscale Stokes-Darcy problems can arise from either the presence of highly oscillatory coefficients within the system or the heterogeneity of the domain; These problems can be very challenging due to the necessary resolution for achieving meaningful results. Constrained by computational cost and prohibitive size, numerous practical problems remain beyond reach through direct simulations. On the other hand, there is currently no effective algorithm for the multiscale problem of Stokes-Darcy with BJS (Beavers-Joseph-Saffman) interface conditions. Thus, it is desirable to develop an effcient computational algorithm to solve multiscale problems without being confined to solving fine scale solutions and establish the corresponding error analysis.

    Prior to presenting our multiscale Stokes-Darcy  multiscale techniques and their application to the Stokes-Darcy problem. In addition to some traditional upscaling techniques, nowadays, various multiscale methods can be employed to solve multi-scale problems. One type of multiscale method involves generating new basis functions, where the aim is to solve the problem on a coarse grid using carefully designed multiscale basis functions. Notable multiscale methods such as multiscale finite method  (MsFEM) by Hou and Wu \cite{hou1997multiscale,chen2003mixed}. MsFEM's applicability broadens to situations where analytical representations of microscopic elements are unavailable, given that the multiscale basis is calculated rather than modeled. Within the last decades, several methods sprung from similar purpose namely, the generalized multiscale finite element method (GMsFEM) \cite{efendiev2013generalized,chung2015mixed},
    the multiscale finite volume method (MsFVM) \cite{hajibeygi2009multiscale,jenny2003multi,efendiev2006accurate},
    the heterogeneous multiscale method (HMM) \cite{weinan2007heterogeneous}, the
    variational multiscale method (VMS) \cite{hughes98},  the multiscale mortar mixed finite
    element method (MMMFEM) \cite{ArPeWY07}, the localized orthogonal method (LOD) \cite{maalqvist2014localization}, the multiscale hybrid-mixed method (MHM) \cite{araya2013multiscale}. Multiscale methods have demonstrated their ability to handle the complexity associated with industry-standard grid representation and flow physics.
    
      Several theoretical and numerical studies
   have been done in the couple years to solve Stokes-Darcy problem. Jun Yao et al. \cite{zhang2016multiscale} presented a multiscale mixed finite element method (MsMFEM) for fluid flow in fractured vuggy media. 
   ABDULLE et al. \cite{abdulle2015adaptive} introduced Darcy-Stokes finite element heterogeneous multiscale method (DS-FE-HMM) in porous media. Girault et al. \cite{girault2014mortar} investigated mortar multiscale numerical methods for coupled Stokes
   and Darcy flows with the Beavers-Joseph-Saffman(BJS) interface condition. Ilona Ambartsumyan et al.\cite{ambartsumyan2020stochastic}  introduced stochastic multiscale flux basis for Stokes-Darcy flows. To date, although there are many multiscale methods available for addressing the Stokes-Darcy problem, MsFEM has not been applied to the Stokes-Darcy problem with BJS interface conditions and conducted theoretical error analysis.
     
     In this paper we propose an Msfem method to solve steady Stokes-Darcy problem with BJS interface condition and derive a fully a priori error analysis.
     Our objective is to propose effective methods that minimize computational efforts when dealing with multiscale phenomena in the Darcy region. 
     While the basis functions of MsFEM have seen widespread use, their application in the context of steady-state Stokes-Darcy problems with BJS interface conditions is notably lacking, let alone the presence of comprehensive error analysis. We use multiscale finite methods in the Darcy region, assuming the presence of multiscale phenomena, and apply standard finite element methods in the Stokes region, with the coupling between the two regions established through an interface. The multiscale finite element basis functions are inspired by Hou's work \cite{hou1997multiscale,chen2003mixed} and generated using parallel methods. In the Stokes region, we employ standard MINI elements for the basis functions. Additionally, we utilize the Robin-Robin algorithm to obtain the final solution. Then, we conducted error analysis in terms of $L^2$ and $H^1$, assuming the Darcy region possesses certain periodic coefficients. Finally, numerical examples will validate the effectiveness of our algorithm and error analysis.

	The rest of the paper is organized as follows. In section \ref{sec:pre}, we introduce the Stokes-Darcy model.
	Next, we introduced the finite element space and the multiscale basis function space.
	 Meanwhile, we also introduced some homogenization principles and certain model results. The Multiscale finite Stokes-Darcy algorithm are presented in Section \ref{sec:offline}.
	 It is mainly divided into two parts: offline and online phase. The corresponding $H^1$ and $L^2$ error estimates is
	shown in Section  \ref{sec:convergence}.
	Numerical experiments are presented in Section \ref{sec:numerical}. 
	
	\section{Preliminaries.}\label{sec:pre}
	\subsection{Stokes-Darcy model with BJS interface condtion}
    Let us consider the following mixed model for coupling a fluid flow and a porous media flow in a bounded domain $\Omega \subset \mathbf{R}^d, d=2,3$. Here $\Omega=\Omega_f \cup \Gamma \cup \Omega_p$, where $\Omega_f$ and $\Omega_p$ are two disjoint, connected and bounded domains occupied by fluid flow and porous media flow and $\Gamma=\bar{\Omega}_f \cap \bar{\Omega}_p$ is the interface. For simplicity, we assume $\partial \Omega_p$ and $\partial \Omega_f$ are smooth enough in the rest of this paper. We denote $\Gamma_f=\partial \Omega_f \cap \partial \Omega, \Gamma_p=\partial \Omega_p \cap \partial \Omega$ and we also denote by $\mathbf{n}_p$ and $\mathbf{n}_f$ the unit outward normal vectors on $\partial \Omega_p$ and $\partial \Omega_f$, respectively. See Fig. \ref{fig:sd} for a sketch.
		\begin{figure}[ht]
		\centering
		\includegraphics[width=3in,height=2in]{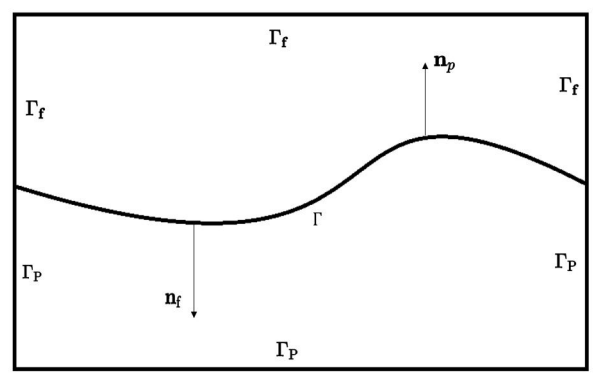}
		\caption\centering{A global domain $\Omega$ consisting of a fluid flow region $\Omega_f$ \protect\\ and a porous media flow region $\Omega_p$ separated by an interface $\Gamma$.}
		\label{fig:sd}
	\end{figure}

      The fluid motion in the fluid region $\Omega_f$ is governed by the Stokes equations
      \begin{equation}\label{Stokes}
       \begin{cases}-\nabla \cdot\left(\mathbb{T}\left(\mathbf{u}_f, p_f\right)\right)=\mathbf{g}_f, & \text { in } \Omega_f, \\ \nabla \cdot \mathbf{u}_f=0, & \text { in } \Omega_f,\end{cases}
      \end{equation}

     where
     $$
     \mathbb{T}\left(\mathbf{u}_f, p_f\right)=-p_f \mathbb{I}+2 v \mathbb{D}\left(\mathbf{u}_f\right), \quad \mathbb{D}\left(\mathbf{u}_f\right)=\frac{1}{2}\left(\nabla \mathbf{u}_f+\nabla^T \mathbf{u}_f\right),
     $$
     are the stress tensor and the deformation rate tensor, $v>0$ is the kinetic viscosity and $\mathbf{g}_f$ is the external force.

     We assume that the porous region possesses multiple scales. The fluid motion in the porous medium region $\Omega_p$ is  governed by

     \begin{equation}\label{DARCY}
      \begin{cases}\nabla \cdot \mathbf{u}_d=g_p, & \text { in } \Omega_p, \\ \mathbf{u}_d=-\mathbb{K}_{\epsilon} \nabla \phi_p, & \text { in } \Omega_p,\end{cases}
     \end{equation}

     where $\mathbb{K}_{\epsilon}$ denotes the hydraulic conductivity in $\Omega_p$, which is a positive symmetric tensor and includs multiscale information where $\epsilon$ is a small parameter and $g_p$ is a source term. The first equation is the saturated flow model and the second equation is the Darcy's law. Here $\phi_p=z+\frac{p_p}{\rho g}$ is the piezometric (hydraulic) head, where $p_p$ represents the dynamic pressure, $z$ the height from a reference level, $\rho$ the density and $g$ the gravitational constant, and $\mathbf{u}_d$ is the flow velocity in the porous medium which is proportional to the gradient of $\phi_p$, namely, the Darcy's law.

      Combining the two equations in (\ref{DARCY}), we get the equation for the piezometric head, which we will refer to simply as the Darcy equation:
     \begin{equation}\label{DARCY2}
     -\nabla \cdot\left(\mathbb{K}_{\epsilon} \nabla \phi_p\right)=g_p, \quad \text { in } \Omega_p .
     \end{equation}

     Eqs. (\ref{Stokes}) and (\ref{DARCY2}) are completed and coupled together by the following boundary conditions:
     $$
     \mathbf{u}_f=0 \quad \text { on } \Gamma_f,  \quad \phi_p=0 \quad \text { on } \Gamma_{p},
     $$
     and the interface conditions on $\Gamma$ :
  	\begin{align}
  	&\mathbf{u}_f\cdot\mathbf{n}_f=\mathbf{u}_p\cdot\mathbf{n}_p=-\mathbb{K}_{\epsilon}\nabla \phi_p\cdot \mathbf{n}_p, \label{mass}\\
  	&-\mathbf{n}\cdot \mathbf{T}(\mathbf{u}_f,p_f)\cdot\mathbf{n}=g(\phi_p-z), \label{momentum}  \\
  	&-\boldsymbol{\tau}_i\cdot \mathbf{T}(\mathbf{u}_f,p_f)\cdot\mathbf{n}_f=\frac{\alpha \nu \sqrt{d}}{\sqrt{\text{trace}(\Pi)}} \mathbf{u}_f\cdot\mathbf{\boldsymbol{\tau}_i}.\label{BJS}
  \end{align} 
  where $\boldsymbol{\tau}_i,i=1,\cdots, d-1$ is an orthonormal basis of the tangential space on $\Gamma$ and $g$ the gravitational acceleration and we assumed that $g=1$ in the following. $ \alpha$ is an experimentally determined parameter and $\Pi$ represents the permeability, which has the following relation with the hydraulic conductivity, $\mathbb{K}_{\epsilon}=\frac{\Pi g}{v}$. The first interface condition (\ref{mass}) describes the mass conservation and the second equation (\ref{momentum}) represents the balance of momentum. The third interface condtion (\ref{BJS}) is called Beavers-Joseph-Saffman condition, which means the tangential components of the normal stress force is proportional to the tangential components of the fluid velocity\cite{beavers1967boundary}.
  
   Furthermore, we assume $\mathbb{K}_{\epsilon}$ is symmetric, perodic, and uniformly elliptic. There exist two constants $\lambda_{\max }>0, \lambda_{\min }>0$ such that $$0<\lambda_{\min }|\mathbf{x}|^2 \leq \mathbb{K}_{\epsilon} \mathbf{x} \cdot \mathbf{x} \leq \lambda_{\max }|\mathbf{x}|^2, \quad \forall \mathbf{x} \in \Omega_p$$  
   
   Also, we assume 
   $$
   \mathbf{g}_f \in \mathbf{L}^2\left(\Omega_f\right), \quad g_p \in L^2\left(\Omega_p\right), \quad \mathbb{K}_{\epsilon} \in L^{\infty}\left(\Omega_p\right)^{d \times d} .
   $$

 For simplicity, we consider 2D problems through the full text. We reserve $\Omega$ for a domain (bounded and open set) with Lipschitz boundary and $d$ for spacial dimension $(d=2)$. The Einstein summation convention is adopted, means summing repeated indexes from 1 to $d$. 
  The Sobolev spaces $W^{k, p}$ and $H^k$ are defined as usual (see e.g., \cite{brenner2008mathematical}) and we abbreviate the norm Sobolev space $H^k(D)$ as $\|\cdot\|_{k, D}$.
 

 Later on we need to introduce some Hilbert spaces
   \begin{equation}
   	\begin{array}{l}
   		\mathbf{V}_f=\left\{\mathbf{v}_f\in \mathbf{H}^1 (\Omega_f):\mathbf{v}_f|_{\Gamma_{f,D}}=0\right\},\\
   		X_p=\left\{\psi\in H^1 (\Omega_p):\psi|_{\Gamma_{p,D}}=0\right\},\\
   		Q_f=L_0^2(\Omega_f)=\left\{q_f\in L^2(\Omega_f):\int_{\Omega_f}	q_f=0	\right\}.\\
   	\end{array}
   \end{equation}
The space $L^2(D)$, where $D=\Omega_{\mathrm{f}}$ or $\Omega_{\mathrm{p}}$, is equipped with the usual $L^2$-scalar product $(\cdot, \cdot)$ and $L^2$-norm $\|\cdot\|_{L^2(D)}$. The spaces $H_{\mathrm{f}}$ and $H_{\mathrm{p}}$ are equipped with the following norms:
$$
\begin{gathered}
\|\nabla u\|_{L^2\left(\Omega_{\mathrm{f}}\right)}=\sqrt{(\nabla u, \nabla u)_{\Omega_{\mathrm{f}}}} \quad \forall u \in \mathbf{V}_f, \\
	\|\nabla \phi\|_{L^2\left(\Omega_{\mathrm{p}}\right)}=\sqrt{(\nabla \phi, \nabla \phi)_{\Omega_{\mathrm{p}}}} \quad \forall \phi \in X_p .
\end{gathered}
$$

  Let us denote 
   \begin{align*}
   	U=\mathbf{V}_f\times X_p.
   \end{align*}
Hence, we use the notational convention that $\underline{\mathbf{u}}=(\mathbf{u}_f,\phi_p)$ and $\underline{\mathbf{v}}=(\mathbf{v}_f,\psi_p)$. They all belong to $\mathbf{U}$: for $\mathbf{g_f} \in \mathbf{L}^2(\Omega_f)$ and $g_p\in L^2(\Omega_p)$, find $(\underline{\mathbf{u}},p_f)\in \mathbf{U}\times Q_f$ such that $\forall(\underline{\mathbf{v}},q_f)\in \mathbf{U}\times Q_f $ and $\mathbf{U}^{\prime}$ is the dual space of $\mathbf{U}, P_\tau(\cdot)$ is the projection onto the local tangential plane that can be explicitly expressed as $P_\tau\left(\mathbf{v}_f\right)=\mathbf{v}_f-\left(\mathbf{v}_f \cdot \mathbf{n}_f\right) \mathbf{n}_f$
   
   \begin{equation}\label{nn}
   	a(\underline{\mathbf{u}},\underline{\mathbf{v}})-(p_f,\nabla\cdot \mathbf{v_f})_{\Omega_f}+(q_f,\nabla\cdot \mathbf{u_f})_{\Omega_f}=\left(\mathbf{F},\underline{\mathbf{v}}\right)_{\mathbf{U}^{\prime}},
   \end{equation}
   where 	
   
   \begin{equation}
   	\begin{aligned}
   		& a(\underline{\mathbf{u}}, \underline{\mathbf{v}})=2 \nu\left(\mathbb{D}\left(\mathbf{u}_f\right), \mathbb{D}\left(\mathbf{v}_f\right)\right)_{\Omega_f}+g\left(\phi_p, \mathbf{v}_f \cdot \mathbf{n}_f\right)_{\Gamma} \\
   		& \quad+\left(\frac{\alpha v \sqrt{d}}{\sqrt{\operatorname{trace}(\boldsymbol{\Pi})}} P_{\boldsymbol{\tau}}\left(\mathbf{u}_f\right), \mathbf{v}_f\right)_{\Gamma} \\
   		& \quad+g\left(\mathbb{K}_{\epsilon} \nabla \phi_p, \nabla \psi_p\right)_{\Omega_p}-g\left(\psi_p, \mathbf{u}_f \cdot \mathbf{n}_f\right)_{\Gamma},\\
   		&\left(\mathbf{F}, \underline{\mathbf{v}}\right)_{\mathbf{U}^{\prime}}=\left(\mathbf{g}_f, \mathbf{v}_f\right)_{\Omega_f}+g\left(g_p, \psi_p\right)_{\Omega_p}+g\left(z, \mathbf{v}_f \cdot \mathbf{n}_f\right)_{\Gamma}.
   	\end{aligned}
   \end{equation}
   
   We know that there exists a positive constant $\beta>0$ such that the following Ladyzhenskaya-Babuš kaBrezzi (LBB) condition holds:
   $$
   \inf _{q_f \in Q_f} \sup _{\mathbf{v}_f \in \mathbf{X}_f} \frac{\left(q_f, \nabla \cdot \mathbf{v}_f\right)_{\Omega_f}}{\left\|q_f\right\|_{Q_f}\left\|\mathbf{v}_f\right\|_{\mathbf{x}_f}} \geq \beta .
   $$  
   \begin{theorem}
   [Proofs in \cite{wilbrandt2019stokes} ]	The weak formulation (\ref{nn}) of Stokes-Dacry problem is well-posed.
   \end{theorem}
  \begin{remark}
	For the purpose of later analysis, we recall some inequalities: $\forall v \in H^1(D)$
	$$
	\begin{aligned}
		& \|v\|_{L^2(\partial D)} \leq c_0\|v\|_{L^2(D)}^{\frac{1}{2}}\|v\|_{H^1(D)}^{\frac{1}{2}}, \\
		& \|v\|_{L^2(\partial D)} \leq c_1\|v\|_{H^1(D)}, \\
		& \|\nabla v\|_{L^2(D)} \leq c_2\|\mathrm{D}(v)\|_{L^2(D)} .
	\end{aligned}
	$$
\end{remark}

	\subsection{Finite element approximation }
  We give the finite element approximation of this model.
For any given small parameter $h>0$, we construct the regular triangulations $\mathcal{T}_h, \mathcal{T}_{f h}$ and $\mathcal{T}_{p h}$ of $\Omega, \Omega_f$ and $\Omega_p$. 

Let $\textbf{V}_{fh} \subset \textbf{V}_{f}$ ,$X_{ph} \subset X_{p}$,and $Q_{fh} \subset Q_{f}$ be finite element spaces such that the space pair $(\textbf{V}_{fh},Q_{fh})$ satisfies the discrete LBB condition : there exists a constant $\beta >0$,independent of mesh ,such that
\begin{align}\label{LBB}
	\underset{0 \neq q_{h}\in Q_{fh}}{inf}  \underset{0 \neq \textbf{v}_{h}\in \textbf{X}_{fh}}{sup}  
	\frac{(q_{h},\nabla \cdot \textbf{v}_{h})_{\Omega_{f}}}{\|\textbf{v}_{h}\|_{1}  \|q_{h}\|_{0}}> \beta.
\end{align}
Here we  choose MINI finite element pair for $(\textbf{V}_{fh}, Q_{fh})$ and multiscale finite element for $V_{ph}$.  Define $$U_{h}=\textbf{V}_{fh}\times X_{ph}.$$

We define Lagrange element space $\mathcal{P}:=\left\{u \in C(\bar{\Omega}), \forall T \in \mathcal{T}_h,\left.u\right|_T \in \mathcal{P}_1(T)\right\}$
and consider a triangulation $\mathcal{T}_{ph}$ and base functions $\left(\psi^i\right)_{1 \leq i \leq  N b_{\text {vertex }}}\in \mathcal{P} $ which consists of the globally continuous on $\mathcal{T}_{ph}$ and affine on each triangle $K \subset \mathcal{T}_{ph}$ functions which satisfy: $\psi^i\left(x_j\right)=\delta_{i j}$,where $\left(x_j\right)_{1 \leq j \leq  N b_{\text {vertex }}}$ are the vertices of the coarse mesh ($ N b_{\text {vertex }} $ the set of interior vertices of the coarse mesh) and $\delta_{i j}$ is the Kronecker symbol. The multiscale finite element basis functions $\left(\eta_{i}^{MsFEM}\right)_{1 \leq i \leq  N b_{\text {vertex }}}$ which take into account the multiscale of the coefficients. More precisely we compute $\eta_{i,K}^{MsFEM}$ , such that for any triangle $K \subset \mathcal{T}_{ph}$,
$$
\left\{\begin{array}{l}
	-\operatorname{div}\left(\mathbb{K}_{\epsilon}\left(\frac{x}{\varepsilon}\right) \nabla \eta_{i,K}^{MsFEM}\right)=0 \text { in } K, \\
	\eta_{i,K}^{MsFEM}=\psi^i \text { on } \partial K .
\end{array}\right.
$$

In practice, we do not have access to $ \eta_{i,K}^{MsFEM}$. We build $ \eta_{i,K}^{MsFEM,h_{finer}}$, an approximation of $ \eta_{i,K}^{MsFEM}$ on a finer embedded grid of mesh size $h_{finer}$, with $\mathcal{P}_1$ FE. See Fig.\ref{fig:basis}.

Therefore, we can get the multiscale finite element basis function $\{\eta_{i}^{MsFEM}=\cup_{K \in \mathcal{K}^h}\eta_{i,K}^{MsFEM}\}$ and a simple observation tells that $\eta_{i}^{MsFEM}$ is locally support. 

\begin{figure}
	\centering
	\includegraphics[width=0.8\linewidth]{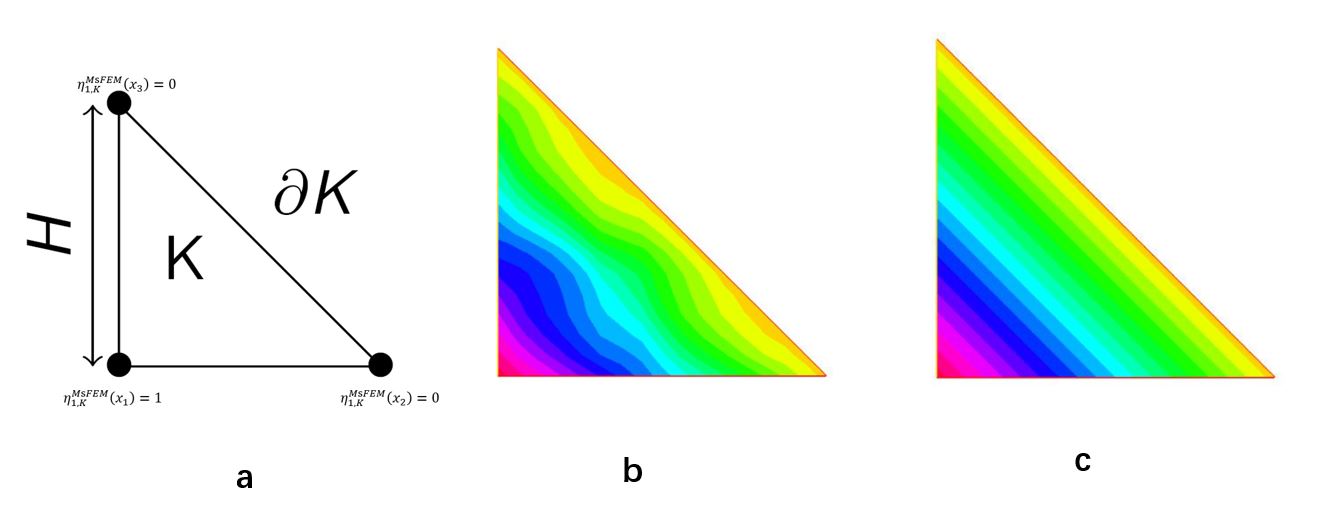}
	\caption{Sketch of MsFEM basis function design in 2D (a), Example of MsFEM basis function (b) and P1 piecewise function (c)}
	\label{fig:basis}
\end{figure}

%
%
%

Then we can introduce 
$$
X_{p h}=\operatorname{span}\left\{\eta_{i}^{MsFEM}: i=1,\cdots,N b_{\text {vertex }}; K \in \mathcal{K}^h\right\} .
$$

%

Therefore, the coupled finite element Galerkin approximation of  reads:

Coupled Finite Element Scheme: find $\underline{\mathbf{u}}_h=\left(\mathbf{u}_{f h}, \phi_{p h}\right) \in \mathbf{U}_h, p_{f h} \in Q_{f h}$ such that for any $\underline{\mathbf{v}}_h=\left(\mathbf{v}_{f h}, \phi_{p h}\right) \in \mathbf{U}_h$ and $q_{f h} \in Q_{f h}$
\begin{equation}\label{weak2}
a\left(\underline{\mathbf{u}}_h, \underline{\mathbf{v}}_h\right)-\left(p_{f h}, \nabla \cdot \mathbf{v}_{f h}\right)_{\Omega_f}+\left(q_{f h}, \nabla \cdot \mathbf{u}_{f h}\right)_{\Omega_f}=\left\langle\mathbf{F}, \underline{\mathbf{v}}_h\right\rangle_{\mathbf{U}^{\prime}}.
\end{equation}


\begin{theorem}
	  [Proofs in \cite{wilbrandt2019stokes} ]	The weak formulation (\ref{weak2}) of Stokes-Dacry problem is well-posed.
\end{theorem}

 \subsection{Homogenization theory and estimation results for the Darcy region}
The Darcy region exhibits multiscale characteristics. In order 
to better elucidate our model, we introduce the model from \cite{ye2020convergence}

\begin{equation}\label{problema}
	\left\{\begin{array}{rl}
		-\operatorname{div}\left(\mathbb{K}_{\epsilon} \nabla\phi_p\right)=f & \text { in } \Omega \\
		\phi_p=0 & \text { on } \Gamma_D \\
		\boldsymbol{n} \cdot \mathbb{K}_{\epsilon} \nabla \phi_p=g & \text { on } \Gamma_N
	\end{array} .\right.
\end{equation}

For simplicity, we assume the source term $f$ belongs to $L^2(\Omega)$, and the boundary term $g \in L^2\left(\Gamma_N\right)$.

Then, we introduce the multiscale expansion technique in deriving homogenized equations.
We look for $\phi_\epsilon(x)$ in the form of asymptotic expansion
\begin{equation}
\phi_\epsilon(x)=\phi_0(x, x / \epsilon)+\epsilon \phi_1(x, x / \epsilon)+\epsilon^2 \phi_2(x, x / \epsilon)+\cdots,
\end{equation}
where the functions $\phi_j(x, y)$ are periodic in $y$ with period 1 .

Here, we  referenced some definitions and theorems from \cite{ye2020convergence}.

\begin{definition}
	A (vector/matrix value) function $f$ is called periodic, if $f(\boldsymbol{x}+\boldsymbol{z})=f(\boldsymbol{x}) \quad \forall \boldsymbol{x} \in \mathbb{R}^d$ and $\forall \boldsymbol{z} \in \mathbb{Z}^d$.
\end{definition}

%

We need an interpolation operator. If the triangulation $\mathcal{T}_h$ is regular, then there exist an interpolation operator $\mathcal{I}: H^1(\Omega) \mapsto \mathcal{P}$ , $\phi_{\epsilon, \mathcal{I}} \in X_{ph}$.

\begin{theorem}\label{ye2020}
	[Proofs in \cite{ye2020convergence}]
	Under some assumptions and $\mathbb{K}_{\epsilon}$ is periodic , for (\ref{problema}) we have:
	$$
	\left|\phi_p-\phi_{\epsilon, \mathcal{I}}\right|_{1, \Omega} \leq C\left\{\epsilon^{1 / 2}+h+\sqrt{\frac{\epsilon}{h}}\right\}\left\|\phi_0\right\|_{2, \Omega}.
	$$
\end{theorem}
	\section{Multiscale finite Stokes-Darcy algorithm}\label{sec:offline}
	
	In this section, we introduce the algorithm for the multiscale Stoke-Darcy equation. The algorithm is divided into two parts, namely the offline part and the online part. The offline part is used to compute the multiscale basis functions for the Darcy region. The online part utilizes the previously calculated multiscale basis functions and employs the Robin-Robin algorithm to obtain the solution of the equation.
	\subsection{Offline phase}
	In the offline phase, we solve the multiscale basis functions in parallel. This approach is adopted to enhance the efficiency of basis function generation. 
	\begin{itemize}
		\item The grid is partitioned to obtain the required coarse mesh $K \in \mathcal{T}_h$.
		\item For every $K \in \mathcal{T}_h$, we build $\mathcal{T}_{}h_{finer}^K$ and
		parallelly solve $-\nabla \cdot\left(\mathbb{K}^\epsilon(x / \epsilon)\right) \nabla \eta_i=0$ in $K, \eta_i=\varphi_i$ on $\partial K$ with $\varphi_i$ the standard $\mathrm{P} 1 $ basis function and store $\eta_i$. See Fig.\ref{fig:para}.
		
		\begin{figure}[htb]
			\centering
			\includegraphics[width=0.3\linewidth]{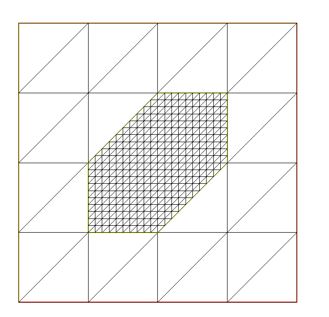}
			\caption{$\mathcal{T}_{h_{finer}}^K$ }
			\label{fig:finer}
		\end{figure}

		\item Compute and store $A_{i, j}^{\text {loc }}=\int_K\left(\mathbb{K}^\epsilon \nabla \eta_i\right)\left(\nabla \eta_j\right)$ and $B_i=B_i+\int_K \eta_i \times f_p$ with $A^{l o c}$ stifness matrix and $B$ the generic RHS.
		\item Output: Assemble and store $A$ the stiffness matrix associted with the new basis
		$
		\eta_H.
		$
	\end{itemize}

\begin{figure}[htb!]
	\centering
	\includegraphics[width=0.9\linewidth]{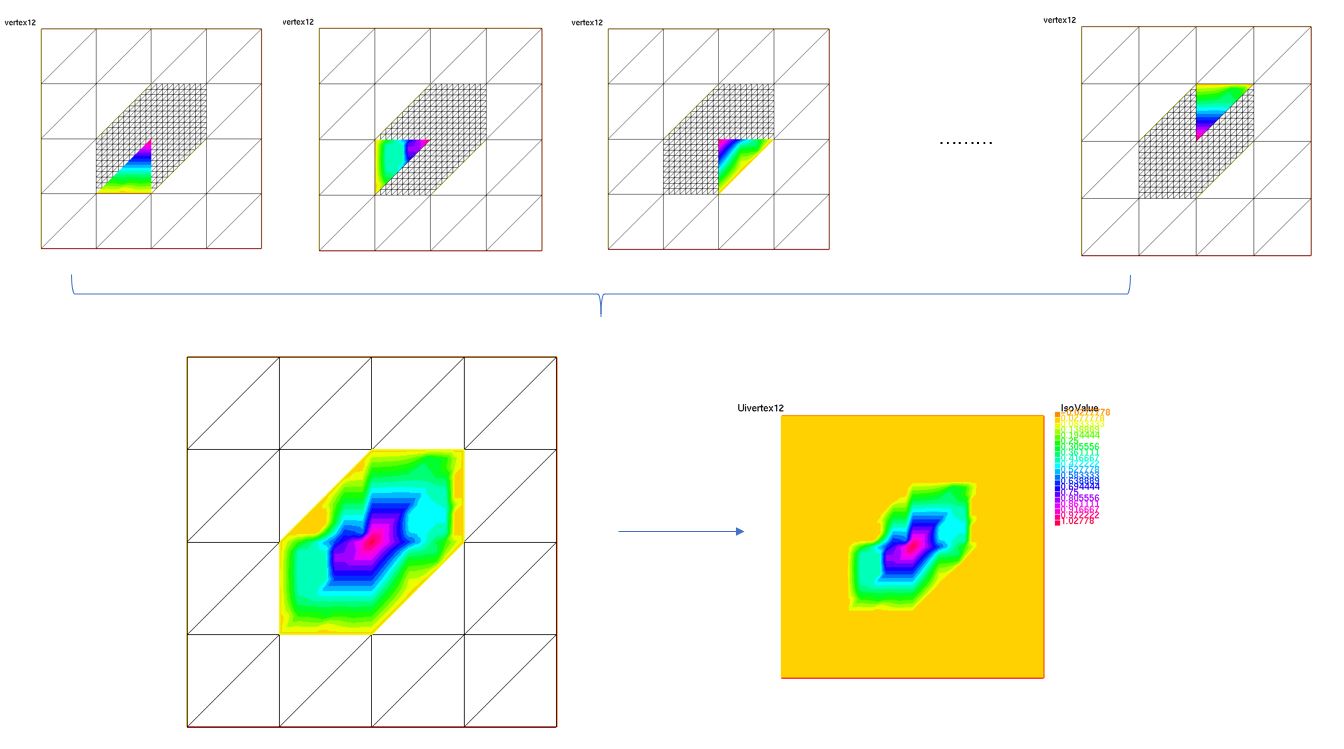}
	\caption{Parallel Efficient Generation of Base Functions}
	\label{fig:para}
\end{figure}

	\subsection{Online phase}
	
	
	In the online phase, we employ the Robin-Robin Stokes-Darcy algorithm. In this process, we utilize the multiscale basis functions generated in the offline stage. This algorithm is a decoupled algorithm, and here are the main steps:

    \begin{itemize}
    	\item Given the initial condtion $\mathbf{u}_{f}^0, p_{f}^0,\varphi_{p}^0$, $\epsilon_{iter}$,$\epsilon_{error}$,$\gamma_{f},\gamma_{p}$.
    	\item For $k=1,2, \ldots$, independently solve the Stokes and Darcy systems with Robin boundary conditions. More precisely, 
    	find $(\mathbf{u}_{fh}^{k+1},p_{fh}^{k+1})$ such that for all
     $\mathbf{v}_{fh}\in\mathbf{V}_{fh}$ and $q_{fh}\in Q_{fh}$ it is 
   
    		\[
    	a_{f}\left(\mathbf{u}_{fh}^k, \mathbf{v}_{fh}\right)+b_{f}\left(\mathbf{v}_{fh}, p_{fh}^k\right)+\gamma_{f}\left(\mathbf{u}_{fh}^k \cdot \mathbf{n}_{f}, \mathbf{v}_{fh} \cdot \mathbf{n}_{f}\right)+\alpha\left( P_\tau \mathbf{u}_{fh}^k, P_\tau \mathbf{v}_{fh}\right) =\left(\eta_{f}^k, \mathbf{v}_{fh} \cdot \mathbf{n}_{f}\right)+\left(\mathbf{g}_f, \mathbf{v}_{fh}\right),  \]
    	\[
    	b_{f}\left(\mathbf{u}_{fh}^k, q_{fh}\right)=0   .
    	\]

    	\item 
    	Find $\varphi_{ph}^{k+1}$ such that for all $\psi_{ph}\in Q_{ph}$ it is
    	
    	\[\gamma_{p} a_{p}\left(\phi_{ph}^k, \psi_{ph}\right)+\left( g \phi_{ph}^k, \psi_{ph}\right)=\left(\eta_{p}^k, \psi_{ph}\right)+\left( g_p,\psi_{ph}\right).\]
    	\item Update $\eta_{p}^{k+1}$ and $\eta_{f}^{k+1}$

    	\[\eta_{f}^{k+1}=a \eta_{p}^k+b g \phi_{ph}^k, \text{where}   \quad a=\frac{\gamma_{f}}{\gamma_{p}}, b=-1-a       
    	\]  
    	\[\eta_{p}^{k+1}=c \eta_{f}^k+d \mathbf{u}_{fh}^k \cdot \mathbf{n}_{f},\text{where}\quad	c=-1 , d=\gamma_{f}+\gamma_{p}
    	\]
    	\item Compute $\epsilon_{iter}=\|\mathbf{u}_{fh}^{k+1}-\mathbf{u}_{fh}^k\|^2+\|p_{fh}^{k+1}-p_{fh}^{k}\|^2+\|\varphi_{ph}^{k+1}-\varphi_{ph}^k\|^2$
    	\item If $\epsilon_{iter}>\epsilon_{error}$, then $k:=k+1$.
    	\item Output: get the soultion $\mathbf{u}_{fh}^{N}, p_{fh}^{N}, \varphi_{ph}^{N}$.
    \end{itemize}
	
	The offline phase involves a parallel method for generating multiscale basis functions. In the online phase, a Robin-Robin online Stokes-Darcy method is employed to obtain the solution. The algorithm offers advantages such as parallel construction of basis functions, which enhances reusability, and ensures high computational accuracy and efficiency. Additionally, the method is easily implementable for generalization and maintenance.

	\section{ Error estimates}\label{sec:convergence}
	
  In this section, we analyze the error estimate of the above coupled multiscale finite element scheme (\ref{weak2}). For later analysis, 	We define the following orthogonal projection $P_{f h}$ from $\mathbf{X}_f$ onto $\mathbf{X}_{f h}$ and $\rho_{f h}$ from $Q_f$ onto $Q_{f h}$ as: for any given $\mathbf{v}_f \in \mathbf{X}_f$ and $p_f\in Q_f$, find $P_{f h} \mathbf{v}_f \in \mathbf{X}_{f h}$ and $\rho_{f h} p_f\in Q_{f h}$  such that
  \begin{equation}\label{projection}
    \begin{aligned}
  	& 2 v\left(\mathbb{D}\left(\mathbf{v}_f-P_{f h} \mathbf{v}_f\right), \mathbb{D}\left(\mathbf{v}_{f h}\right)\right)_{\Omega_f} \\
  	& \quad+\left(\frac{\alpha v \sqrt{d}}{\sqrt{\operatorname{trace}(\boldsymbol{\Pi})}} P_\tau\left(\mathbf{v}_f-P_{f h} \mathbf{v}_f\right), \mathbf{v}_{f h}\right)_{\Gamma}+\left(p_f-\rho_{f h} p_f,\nabla\cdot \mathbf{v}_{f h} \right)_{\Omega_f}=0 \quad \forall \mathbf{v}_{f h} \in \mathbf{X}_{f h},\\
  	&\left(\nabla \cdot (\mathbf{u}_{f}-P_{f h} \mathbf{u}_{f}),q_{f h}\right)=0 \quad \forall q_{f h} \in \mathbf{Q}_{f h}.
  \end{aligned}
  \end{equation}

  For this projection $\rho_{f h}$ and the projection $P_{f h}$ from $\mathbf{X}_f$ onto $\mathbf{X}_{f h}$ in the previous section, we make the following assumption: for any given $\mathbf{v}_f \in \mathbf{H}^2\left(\Omega_f\right) \cap \mathbf{X}_f$ and $q_f \in H^1\left(\Omega_f\right)$, there holds
  \begin{equation}\label{touyin}
  	\left\|\mathbb{D}\left(\mathbf{v}_f-P_{f h} \mathbf{v}_f\right)\right\|_{\Omega_{f}}+\left\|q_f-\rho_{f h} q_f\right\|_{\Omega_{f}} \leq c h \left(\|\mathbf{v}_f\|_{2,\Omega_{f}}+\|q_f\|_{1,\Omega_{f}}\right).
  \end{equation}
  
  	Suppose the weak solution $\left(\underline{\mathbf{u}}, p_f\right)$ of the mixed Stokes/Darcy problem is local $H^2$-regular, that is $\mathbf{u}_f \in \mathbf{H}^2\left(\Omega_f\right) \cap \mathbf{V}_f, \quad p_f \in H^1\left(\Omega_f\right),\quad \phi_p \in H^2\left(\Omega_p\right) \cap X_p$ and $\phi_0 \in H^2\left(\Omega_p\right) \cap X_p$
  
	\begin{theorem}[$H^1$ norm]\label{H1norm}
		Assume that $\mathbb{K}_{\epsilon}$ is periodic, uniformly elliptic and symmtric. Let $\Omega$ be a bounded Lipschitz domain in $\mathbb{R}^2$, and let $(\mathbf{u}_f,\phi_p)$ and $(\mathbf{u}_{fh},\phi_{ph})$ be the solutions of Problems (\ref{nn}) and (\ref{weak2}), respectively.
		
 			\begin{equation}\label{H1new}
 		\begin{aligned}
 			\sqrt{\nu}\left\|\mathbb{D}\left(\mathbf{u}_f-\mathbf{u}_{f h}\right)\right\|_{\Omega_f}+\frac{\sqrt{g}}{2}\left\|\mathbb{K}_{\epsilon}^{\frac{1}{2}} \nabla (\phi_p-\phi_{p h})\right\|_{\Omega_p} &\leq \hat{C}_1\left\{\epsilon^{1 / 2}+h+\sqrt{\frac{\epsilon}{h}}\right\}\left\|\phi_0\right\|_{2, \Omega_p}\\				
 			&+\hat{C}_2h\left(\|\mathbf{u}_f\|_{2,\Omega_{f}}+\|p_f\|_{1,\Omega_{f}}\right),
 		\end{aligned}
 	\end{equation}
		where $\hat{C_1},\hat{C_2}$ depends on $\lambda_{\Max},\lambda_{\Min},g,\nu$.
\end{theorem}

	If we denote

	$$
	\begin{aligned}
		& \mathbf{u}_f-\mathbf{u}_{f h}=\left(\mathbf{u}_f-P_{f h} \mathbf{u}_f\right) +\left(P_{f h} \mathbf{u}_f-\mathbf{u}_{f h}\right)=\hat{\mathbf{e}}_f+\mathbf{e}_{f h},\\
		& \phi_p-\phi_{p h}=\left(\phi_p-\phi_{\epsilon, \mathcal{I}}\right) +(\phi_{\epsilon, \mathcal{I}}-\phi_{p h})=\hat{e}_p+e_{p h},\\
		&p_f-p_{f h}=(p_f-\rho_{f h} p_f)+(\rho_{f h} p_f-p_{f h})=\hat{e}_{\theta}+e_{\theta}.
	\end{aligned}
	$$
	
	\begin{proof}
		Recall the weak formulation and its finite element scheme
		\begin{equation}\label{weakquan}
			\begin{aligned}
				& \left\{\begin{array}{l}
					a(\underline{\mathbf{u}}, \underline{\mathbf{v}})-\left(p_f, \nabla \cdot \mathbf{v}_f\right)_{\Omega_f}=\langle F, \underline{\mathbf{v}}\rangle, \quad \forall\left(\underline{\mathbf{v}}, q_f\right) \in \mathbf{U} \times Q_f, \\
					\left(q_f, \nabla \cdot \mathbf{u}_f\right)_{\Omega_f}=0, \quad \forall q_f \in Q_f.
				\end{array}\right. \\
				& \left\{\begin{array}{l}
					a\left(\underline{\mathbf{u}}_h, \mathbf{v}_h\right)-\left(p_{f h}, \nabla \cdot \mathbf{v}_{f h}\right)_{\Omega_f}+\left(q_{f h}, \nabla \cdot \mathbf{u}_{f h}\right)_{\Omega_f}=\left\langle F, \underline{\mathbf{v}_h}\right\rangle \quad \forall\left(\underline{\mathbf{v}_h}, q_{f h}\right) \in \mathbf{U}_h \times Q_{f h}, \\
					\left(q_{f h}, \nabla \cdot \mathbf{u}_{f h}\right)_{\Omega_f}=0, \quad \forall q_{f h} \in Q_{f h}.
				\end{array}\right.
			\end{aligned}
		\end{equation}
		where 	
		\begin{equation}
			\begin{aligned}
				& a(\underline{\mathbf{u}}, \underline{\mathbf{v}})=2 \nu\left(\mathbb{D}\left(\mathbf{u}_f\right), \mathbb{D}\left(\mathbf{v}_f\right)\right)_{\Omega_f}+g\left(\phi_p, \mathbf{v}_f \cdot \mathbf{n}_f\right)_{\Gamma} \\
				& \quad+\left(\frac{\alpha v \sqrt{d}}{\sqrt{\operatorname{trace}(\boldsymbol{\Pi})}} P_{\boldsymbol{\tau}}\left(\mathbf{u}_f\right), \mathbf{v}_f\right)_{\Gamma} \\
				& \quad+g\left(\mathbb{K} \nabla \phi_p, \nabla \psi_p\right)_{\Omega_p}-g\left(\psi_p, \mathbf{u}_f \cdot \mathbf{n}_f\right)_{\Gamma} 
			\end{aligned}
		\end{equation}

		Then,
		\begin{equation}
			\begin{aligned}
				&2 \nu\left(\mathbb{D}\left(\mathbf{u}_f-\mathbf{u}_{fh}\right), \mathbb{D}\left(\mathbf{v}_{f h}\right)\right)_{\Omega_f}+g\left(\phi_p-\phi_{p h}, \mathbf{v}_{f h} \cdot \mathbf{n}_f\right)_{\Gamma} \\
				& \quad+\left(\frac{\alpha v \sqrt{d}}{\sqrt{\operatorname{trace}(\boldsymbol{\Pi})}} P_{\boldsymbol{\tau}}\left(\mathbf{u}_f-\mathbf{u}_{fh}\right), \mathbf{v}_{fh}\right)_{\Gamma} \\
				& \quad+g\left(\mathbb{K} \nabla (\phi_p-\phi_{p h}), \nabla \psi_{p h}\right)_{\Omega_p}-g\left(\psi_{p h}, \mathbf{u}_f \cdot \mathbf{n}_f-\mathbf{u}_{fh} \cdot \mathbf{n}_{fh}\right)_{\Gamma}\\
				&-\left(p_f-p_{f h},\nabla\cdot\mathbf{v}_{fh} \right)_{\Omega_f}+(q_{f h},\nabla\cdot (\mathbf{u}_f-\mathbf{u}_{f h}))_{\Omega_f}=0.
			\end{aligned}
		\end{equation}
		Using (\ref{projection}) and (\ref{weakquan}), it   deduces:
		\begin{equation}\label{1}
			\begin{aligned}
				&2 \nu\left(\mathbb{D}\left(\mathbf{e}_{f h}\right), \mathbb{D}\left(\mathbf{v}_{f h}\right)\right)_{\Omega_f}+g\left(\hat{e}_p+e_{p h}, \mathbf{v}_{f h} \cdot \mathbf{n}_f\right)_{\Gamma} \\
				& \quad+\left(\frac{\alpha v \sqrt{d}}{\sqrt{\operatorname{trace}(\boldsymbol{\Pi})}} P_{\boldsymbol{\tau}}\left(\mathbf{e}_{f h}\right), \mathbf{v}_{fh}\right)_{\Gamma} \\
				& \quad+g\left(\mathbb{K} \nabla (\hat{e}_p+e_{p h}), \nabla \psi_{p h}\right)_{\Omega_p}-g\left(\psi_{p h}, (\hat{\mathbf{e}}_f+\mathbf{e}_{f h}) \cdot \mathbf{n}_f\right)_{\Gamma}=0.
			\end{aligned}
		\end{equation}
		Taking $\mathbf{v}_{f h}=\mathbf{e}_{f h}, \psi_{p h}=e_{p h} $ into (\ref{1}), we can get
		\begin{equation}
			\begin{aligned}
				&2 \nu\left(\mathbb{D}\left(\mathbf{e}_{f h}\right), \mathbb{D}\left(\mathbf{e}_{f h}\right)\right)_{\Omega_f}+g\left(\hat{e}_p+e_{p h}, \mathbf{e}_{f h} \cdot \mathbf{n}_f\right)_{\Gamma} \\
				& \quad+\left(\frac{\alpha v \sqrt{d}}{\sqrt{\operatorname{trace}(\boldsymbol{\Pi})}} P_{\boldsymbol{\tau}}\left(\mathbf{e}_{f h}\right), \mathbf{e}_{fh}\right)_{\Gamma} \\
				& \quad+g\left(\mathbb{K} \nabla (\hat{e}_p+e_{p h}), \nabla e_{p h}\right)_{\Omega_p}-g\left(e_{p h}, (\hat{\mathbf{e}}_f+\mathbf{e}_{f h}) \cdot \mathbf{n}_f\right)_{\Gamma}=0.
			\end{aligned}
		\end{equation}
		
		Then
		\begin{equation}
			\begin{aligned}
				&2 \nu\left(\mathbb{D}\left(\mathbf{e}_{f h}\right), \mathbb{D}\left(\mathbf{e}_{f h}\right)\right)_{\Omega_f}+g\left(e_{p h}, \mathbf{e}_{f h} \cdot \mathbf{n}_f\right)_{\Gamma} \\
				& \quad+\left(\frac{\alpha v \sqrt{d}}{\sqrt{\operatorname{trace}(\boldsymbol{\Pi})}} P_{\boldsymbol{\tau}}\left(\mathbf{e}_{f h}\right), \mathbf{e}_{fh}\right)_{\Gamma} \\
				& \quad+g\left(\mathbb{K} \nabla (e_{p h}), \nabla e_{p h}\right)_{\Omega_p}-g\left(e_{p h}, (\mathbf{e}_{f h}) \cdot \mathbf{n}_f\right)_{\Gamma}\\
				&=-g\left(\hat{e}_p, \mathbf{e}_{f h} \cdot \mathbf{n}_f\right)_{\Gamma}-g\left(\mathbb{K} \nabla (\hat{e}_p), \nabla e_{p h}\right)_{\Omega_p}+g\left(e_{p h}, \hat{\mathbf{e}}_f \cdot \mathbf{n}_f\right)_{\Gamma}.
			\end{aligned}
		\end{equation}
		and we can get 
		\begin{equation}
			\begin{aligned}
				&2 v\left\|\mathbb{D}\left(\mathbf{e}_{f h}\right)\right\|_{\Omega_f}^2+\frac{g}{2}\left\|\mathbb{K}^{\frac{1}{2}} \nabla e_{p h}\right\|_{\Omega_p}^2+\underbrace{ \left(\frac{\alpha v \sqrt{d}}{\sqrt{\operatorname{trace}(\boldsymbol{\Pi})}} P_{\boldsymbol{\tau}}\left(\mathbf{e}_{fh}\right), \mathbf{e}_{fh}\right)_{\Gamma}}_{\text{\ding{192}}}\\
				&\leq \underbrace{\lvert g\left(\hat{e}_p, \mathbf{e}_{f h} \cdot \mathbf{n}_f\right)_{\Gamma}\rvert}_{\text{\ding{193}}}+\underbrace{\lvert g\left(\mathbb{K} \nabla (\hat{e}_p), \nabla e_{p h}\right)_{\Omega_p}\rvert}_{\text{\ding{194}}}+
				\underbrace{\lvert g\left(e_{p h}, \hat{\mathbf{e}}_f \cdot \mathbf{n}_f\right)_{\Gamma}\rvert}_{\text{\ding{195}}}		
			\end{aligned}
		\end{equation}
		For \ding{192} :
		\begin{equation}
			\begin{aligned}
				\left(\frac{\alpha v \sqrt{d}}{\sqrt{\operatorname{trace}(\boldsymbol{\Pi})}} P_{\boldsymbol{\tau}}\left(\mathbf{e}_{fh}\right), \mathbf{e}_{fh}\right)_{\Gamma}&=\left(\frac{\alpha v \sqrt{d}}{\sqrt{\operatorname{trace}(\boldsymbol{\Pi})}} \left(\mathbf{e}_{fh}\cdot\tau\right), \mathbf{e}_{fh}\cdot\tau\right)_{\Gamma}\\
				&=\frac{\alpha v \sqrt{d}}{\sqrt{\operatorname{trace}(\boldsymbol{\Pi})}}\|\mathbf{e}_{fh}\cdot\tau\|^2_{\Gamma}\geq 0.
			\end{aligned}
		\end{equation}

		Then, we estimate the three terms on the right hand side of the above inequality.
		
		For \ding{193} :
		\begin{equation}
			\begin{aligned}
				g\left|\left(\hat{e}_p, \mathbf{e}_{f h} \cdot \mathbf{n}_f\right)_{\Gamma}\right| & \leq g\left\|\hat{e}_p\right\|_{L^2(\Gamma)}\left\|\mathbf{e}_{f h}\right\|_{\mathbf{L}^2(\Gamma)} \leq {C g}\left\| \nabla \hat{e}_p\right\|_{\Omega_p}\left\|\mathbb{D}\left(\mathbf{e}_{f h}\right)\right\|_{\Omega_f} \\
				& \leq \frac{C^2 g^2}{4  \nu}\left\|\nabla \hat{e}_p\right\|_{\Omega_p}^2+\nu\left\|\mathbb{D}\left(\mathbf{e}_{f h}\right)\right\|_{\Omega_f}^2 .
			\end{aligned}
		\end{equation}
		
		For \ding{194} :
		\begin{equation}
			\begin{aligned}
				\lvert g\left(\mathbb{K} \nabla (\hat{e}_p), \nabla e_{p h}\right)_{\Omega_p}\rvert&\leq g\lambda_{\Max}^{\frac{1}{2}}\left\| \nabla \hat{e}_p\right\|_{\Omega_p}\left\|\mathbb{K}^{\frac{1}{2}}\nabla e_{p h}\right\|_{\Omega_p}\\
				& \leq g{\lambda_{\Max}^{\frac{1}{2}}}\left\| \nabla \hat{e}_p\right\|_{\Omega_p}\left\|\mathbb{K}^{\frac{1}{2}}\nabla e_{p h}\right\|_{\Omega_p}\\
				& \leq 2g{\lambda_{\Max}}\left\| \nabla \hat{e}_p\right\|^2_{\Omega_p}+\frac{g}{8}\left\|\mathbb{K}^{\frac{1}{2}}\nabla e_{p h}\right\|^2_{\Omega_p}.
			\end{aligned}
		\end{equation}

		For \ding{195} :
		\begin{equation}
			\begin{aligned}
				\lvert g\left(e_{p h}, \hat{\mathbf{e}}_f \cdot \mathbf{n}_f\right)_{\Gamma}\rvert &\leq 
				g\|e_{ph}\|_{\Gamma}\|\hat{\mathbf{e}}_f \cdot \mathbf{n}_f\|_{\Gamma}\\
				&\leq \frac{gC}{\lambda_{\Min}^{\frac{1}{2}}}\|\mathbb{K}^{\frac{1}{2}}\nabla e_{ph}\|_{\Omega_p}\|\hat{\mathbf{e}}_f
				\|_{1,\Omega_f}\\
				&\leq \frac{g}{8}\|\mathbb{K}^{\frac{1}{2}}\nabla e_{ph}\|_{\Omega_p}^2+\frac{2gC^2}{\lambda_{\Min}}\|\hat{\mathbf{e}}_f
				\|_{1,\Omega_f}^2.
			\end{aligned}
		\end{equation}
	Combining the above bounds and using (\ref{touyin}) and Theorem \ref{ye2020}, it yields
			\begin{equation}\label{H1}
					\begin{aligned}
							\ \nu\left\|\mathbb{D}\left(\mathbf{u}_f-\mathbf{u}_{f h}\right)\right\|_{\Omega_f}^2+\frac{g}{4}\left\|\mathbb{K}_{\epsilon}^{\frac{1}{2}} \nabla (\phi_p-\phi_{p h})\right\|_{\Omega_p}^2 &\leq\left( \frac{C^2 g^2}{4  \nu}+2g{\lambda_{\Max}}\right)\left\{\epsilon^{1 / 2}+h+\sqrt{\frac{\epsilon}{h}}\right\}^2\left\|\phi_0\right\|_{2, \Omega_p}^2\\				
							&+\frac{2g}{\lambda_{\Min}}Ch^2\left(\|\mathbf{u}_f\|_{2,\Omega_{f}}+\|p_f\|_{1,\Omega_{f}}\right)^2.
						\end{aligned}
				\end{equation} 
	Then, it is easily get the final result (\ref{H1new}).
	
	\end{proof}

\begin{remark}
	Given that  $\|\mathbf{u}_f\|_{2,\Omega_{f}}$, $\|p_f\|_{1,\Omega_{f}}$ and $\left\|\phi_0\right\|_{2,\Omega_{p}}$ are bounded, the $H^1$ norm can be controlled by  $\epsilon^{1 / 2}+h+\sqrt{\frac{\epsilon}{h}}$.
\end{remark}

Let us introduce the formal adjoint problem (\cite{hou2019solution}): find $\underline{\mathbf{w}}=\left(\mathbf{w}_f, \xi_p\right) \in \mathbf{V}$ such that

\begin{equation}
	a(\underline{\mathbf{v}}, \underline{\mathbf{w}})=(f,\underline{\mathbf{v}}) \quad \forall \underline{\mathbf{v}}=\left(\mathbf{v}_f, \psi_p\right) \in \mathbf{V} ,
\end{equation}
where $f:=(\mathbf{u}_f-\mathbf{u}_{f h},\phi_p-\phi_{p h})$.

\begin{equation}
	a(\underline{\mathbf{v}}, \underline{\mathbf{w}})=\left(\mathbf{u}_f-\mathbf{u}_{f h}, \mathbf{v}_f\right)_{\Omega_f}+g\left(\phi_p-\phi_{p h}, \psi_p\right)_{\Omega_p} \quad \forall \underline{\mathbf{v}}=\left(\mathbf{v}_f, \psi_p\right) \in \mathbf{V} .
\end{equation}
 We assume that $\underline{\mathbf{w}}$ has the regularity estimate:
 \begin{equation}
 	\|\underline{\mathbf{w}}\|_{H^2}\leq C_R \|f\|_{L^2},
 \end{equation}
and 
\begin{equation}
	\|\underline{\mathbf{w}}_0\|_{H^2}\leq C_R \|f\|_{L^2},
\end{equation}
where $\underline{\mathbf{w}}_0=\left(\mathbf{w}_f, \xi_{p,0}\right)$ and $\xi_{p,0}$ is the $O(\epsilon^0)$  asymptotic expansion of  $\xi_{p}$.
     		\begin{theorem}[L2 norm]
     			{
     	Assume that $\mathbb{K}^{\epsilon}$ is periodic, uniformly ellpitic and symmtric. Let $\Omega$ be a bounded Lipschitz domain in $\mathbb{R}^2$, and let $(\mathbf{u}_f,\phi_p)$ and $(\mathbf{u}_{fh},\phi_{ph})$ be the solutions of Problems (\ref{nn}) and (\ref{weak2}), respectively.
     	
     	\begin{equation}\label{L2}
     		\begin{aligned}
     			\left\|\left(\mathbf{u}_f-\mathbf{u}_{f h}\right)\right\|_{\Omega_f}+{g}\left\| (\phi_p-\phi_{p h})\right\|_{\Omega_p} &\leq\tilde{C_1}\left\{\epsilon +h^2+\frac{\epsilon}{h}\right\}\left\|\phi_0\right\|_{2, \Omega}\\				
     			&+\tilde{C_2}\left(h^2+\sqrt{\epsilon h}\right)\left(\|\mathbf{u}_f\|_{2,\Omega_{f}}+\|p_f\|_{1,\Omega_{f}}\right),
     		\end{aligned}
     	\end{equation}
     	where $\tilde{C_1},\tilde{C_2}$ depends on $\lambda_{\Max},\lambda_{\Min},g,\nu$.
     }
 \end{theorem}

\begin{proof}
	\begin{equation}
		\begin{aligned}
			\left\|\underline{\mathbf{u}}-\underline{\mathbf{u}}_h\right\|_{\Omega}^2 & =\left({\mathbf{u}}-{\mathbf{u}}_h, {\mathbf{u}}-{\mathbf{u}}_h\right)_{\Omega_f}+g\left(\phi_p-\phi_{p h}, \phi_p-\phi_{p h}\right)_{\Omega_p} \\
			& =a\left(\underline{\mathbf{u}}-\underline{\mathbf{u}}_h, \underline{\mathbf{w}}\right) \\
			& =a\left(\underline{\mathbf{u}}-\underline{\mathbf{u}}_h, \underline{\mathbf{w}}-\mathcal{I}^h \underline{\mathbf{w}}\right) \\
			& \leq C\left\|\underline{\mathbf{u}}-\underline{\mathbf{u}}_h\right\|_{1,\Omega} \| \underline{w}-\mathcal{I}^h \underline{w}\|_{{H^{1}}(\Omega)}\\
			&\leqslant C\left\|\underline{\mathbf{u}}-\underline{\mathbf{u}}_h\right\|_{1,\Omega} \|\sqrt{\left(\left(\nabla\left(w_f-\mathcal{I}^h w_f\right), \nabla\left(w_f-\mathcal{I}^h w_f\right)\right)+g\left(\mathbb{K} \nabla\left(\xi_p-\mathcal{I}^h \xi_p\right),\nabla\left(\xi_p-\mathcal{I}^h \xi_p\right)\right)\right.}\\
			&\leqslant C_1\left\|\underline{\mathbf{u}}-\underline{\mathbf{u}}_h\right\|_{1,\Omega}\left[\left\|w_f-\mathcal{I}^h w_f\right\|_{1,\Omega_f}+\sqrt{g} \lambda_{\lambda_{{\Max }}}^{\frac{1}{2}}\left\|\xi_p-\mathcal{I}^h \xi_p\right\|_{1,\Omega_p}\right]\\
			&\leqslant C_2\left\|\underline{\mathbf{u}}-\underline{\mathbf{u}}_h\right\|_{1,\Omega}\left[h\left\|w_f\right\|_{2,\Omega_f}+\sqrt{g} \lambda_{\text {max }}^{\frac{1}{2}}\left(\varepsilon^{\frac{1}{2}}+h+\sqrt{\frac{\varepsilon}{h}}\right)\left\|\xi_{p,0}\right\|_{2,\Omega_p}\right]\\
			&\leqslant C_3\left\|\underline{\mathbf{u}}-\underline{\mathbf{u}}_h\right\|_{1,\Omega}
			\left(\varepsilon^{\frac{1}{2}}+h+\sqrt{\frac{\varepsilon}{h}}\right)\left(\left\|w_f\right\|_{2,\Omega_f}+\left\|\xi_{p,0}\right\|_{2,\Omega_p}\right)\\
			&\leqslant C_3\left\|\underline{\mathbf{u}}-\underline{\mathbf{u}}_h\right\|_{1,\Omega}\left(\varepsilon^{\frac{1}{2}}+h+\sqrt{\frac{\varepsilon}{h}}\right)\left\|\underline{\mathbf{u}}-\underline{\mathbf{u}}_h\right\|_{\Omega}.
		\end{aligned}
	\end{equation}

Therefore, 
\begin{equation}
	\left\|\underline{\mathbf{u}}-\underline{\mathbf{u}}_h\right\|_{\Omega}\leqslant C_3\left\|\underline{\mathbf{u}}-\underline{\mathbf{u}}_h\right\|_{1,\Omega}\left(\varepsilon^{\frac{1}{2}}+h+\sqrt{\frac{\varepsilon}{h}}\right).
\end{equation}

 Using Theorem \ref{H1norm}, it yields the final result (\ref{L2})
\end{proof}
	
\begin{remark}
	Given that  $\|\mathbf{u}_f\|_{2,\Omega_{f}}$, $\|p_f\|_{1,\Omega_{f}}$ and $\left\|\phi_0\right\|_{2,\Omega_{p}}$ are bounded, the $L^2$ norm can be controlled by $\epsilon +h^2+\frac{\epsilon}{h}$.
\end{remark}

	\section{Numerical experiments}\label{sec:numerical}
	In this section, we study the accuracy of the multiscale method through numerical computations. The model problem is solved using the multiscale method with base functions defined by linear
    condition (Msfem). Since it is very diffcult to construct a test problem with exact solution and suffcient generality, we use resolved numerical solutions in place of exact solutions. The numerical results are compared with the theoretical analysis. Additionally, we present a numerical example illustrating the method's application to more complex problems involving separable scales, nonseparable scales and high-contrast case.

	    \subsection{Implementation}
	    
	    We outline the implementation here and define some notation to be frequently used below. For convenience, we assume that the multiscale Darcy region is situated on the unit square. Let $N$ be the number of elements in the $x$ and $y$ directions. The size of mesh is thus $h=1 / N$. To compute the base functions, each element is discretized into $M \times M$ subcell elements with size $h_{finer}=h / M$. Triangular elements are used in all numerical tests.
	    
	    In the offline phase, we utilize P1 elements to compute basis functions within each element.  During this process, the Message Passing Interface(MPI) is employed to enhance efficiency, and the stiffness matrix and right-hand side are stored for subsequent online computations.
	   
	   Because it is very difficult to construct a genuine $2 \mathrm{D}$ steady Stokes-Darcy model with an exact solution, reference solutions are used as the exact solutions for the test problems. In all numerical examples below, the resolved solutions are obtained using standard FEM. We use the numerical computation libraries MUMPS and PETSc within FreeFEM++ to calculate the reference solution.
	   Due to the multiscale nature of the permeability in our Darcy region, we employ a finer mesh within the Darcy region when computing the reference solution. We have computed a reference solution of this numerical example using Talor-Hood finite element and P2 element with a mesh of size $H_D=1 / 2048$ (in Darcy region), $H_S=1 / 512$ (in Stokes region). The purpose of computing the reference solution is to obtain the error.
	   To enhance computational speed, we utilize a domain decomposition parallel approach for computing the reference solution. We define the reference solution $\underline{\mathbf{u}}^{\text{reference}}=(\mathbf{u}_f^{\text{reference}},\phi_p^{\text{reference}})$.
	   
	   To facilitate the comparison among different schemes, we use the following shorthands: FEM-FEM stands for  MINI elements in the Stokes region and P1 elements in the Darcy region and FEM-MsFEM stands for MINI elements in the Stokes region and multiscale basis function in the Darcy region. We define the numerical solution $\underline{\mathbf{u}}^{\text{numerical }}=(\mathbf{u}_f^{\text{numerical }},\phi_p^{\text{numerical }})$.

	    \subsection{Numerical result}
	    We define some norms to demonstrate the error.
	    \begin{equation*}
	    	\left\|e_{\mathbf{u}_f}\right\|_{L^2}=\left\|\mathbf{u}_f^{\text{reference}}-\mathbf{u}_f^{\text{numerical}}\right\|_{L^2(\Omega_{f})},	\left\|e_{\mathbf{u}_f}\right\|_{H^1}=\left\|\mathbf{u}_f^{\text{reference}}-\mathbf{u}_f^{\text{numerical}}\right\|_{H^1(\Omega_{f})}
	    \end{equation*}
	    \begin{equation*}
	    	\left\|e_{\phi_p}\right\|_{L^2}=\left\|\phi_p^{\text{reference}}-\phi_p^{\text{numerical}}\right\|_{H^1(\Omega_{p})},\left\|e_{\phi_p}\right\|_{L^2}=\left\|\phi_p^{\text{reference}}-\phi_p^{\text{numerical}}\right\|_{H^1(\Omega_{p})}
	    \end{equation*}
        \subsubsection{Example 1}
        In this example, we we consider nonseparable scales and solve (\ref{DARCY2}) with
        $$
      \mathbb{K}_{\epsilon}=\frac{1}{(2+P \sin (2 \pi x / \epsilon))(2+P \sin (2 \pi y / \epsilon))},
        $$
        where $P$ is a parameter controlling the magnitude of the oscillation. We take $P=1.8$ in this example. The right hand side function $\mathbf{g}_f(x,y),g_p(x,y)$ is zero        On $\partial \Omega_p$, we impose $\phi=0$ and on the $\partial \Omega_f$, $u_f=[\sin (\pi x), 0],\text { on }(0,1) \times\{2\}$,
        $u_f=[0,0] , \text { on }\{0\} \times(1,2) \cup\{1\} \times(1,2)$.
        We set all physical parameters $\nu ,\alpha, g$ equal to 1. The Robin-Robin coefficients are set to $\gamma_{f}=0.1,\gamma_{p}=1$.
        
        \begin{figure}[htb]
        	\centering
        	\includegraphics[width=0.5\linewidth]{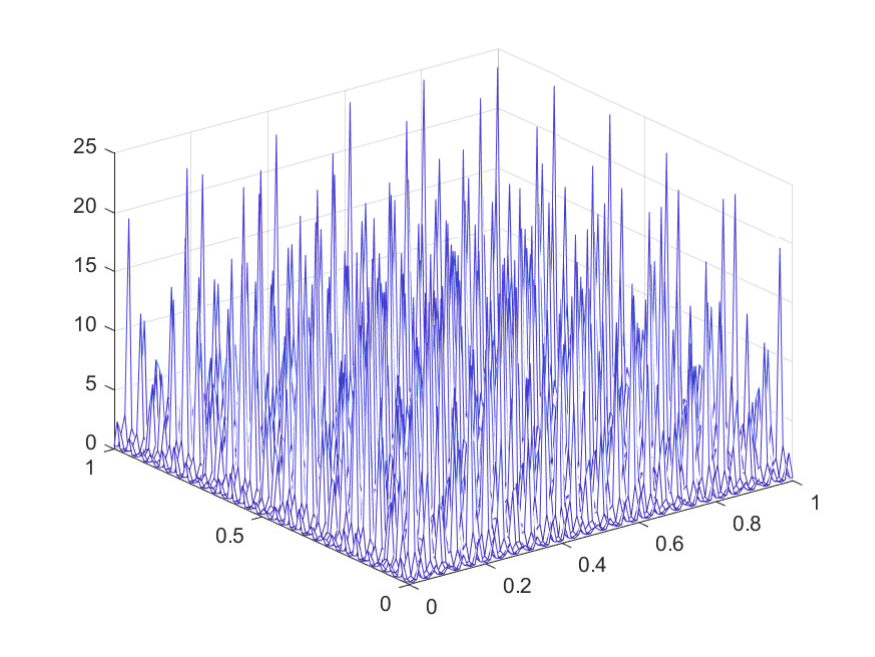}
        	\caption{$\mathbb{K}_{\epsilon}$ in Example 1}
        	\label{fig:k1}
        \end{figure}

       The result of FEM-MsFEM error with varying $h$ but fixed $\epsilon=0.008$ and $h_{finer}=1/2048$ are shown in Table \ref{tab:ex1msfem}. When observing the errors in the Darcy part, we first fix the small parameter $\varepsilon$ and gradually refine the grid. As $h$ decreases and $h \gg \varepsilon$, we can observe that the convergence order of the $L^2$ error for $e_{\phi_p}$ degrades from around $+2$ to around $-1$. This effectively validates the term O($h^2$) to O($\varepsilon/h$) in the $L^2$ error analysis. Additionally, the convergence order of the $H^1$ error degrades from $+1$ to $-0.5$, providing a good validation of the term O($h$) to O($\sqrt{\varepsilon/h}$) in the $L^2$ error analysis. 
       Furthermore, when observing the errors $e_{\mathbf{u}_f}$ in the Stokes part, we use standard MINI elements in the Stokes region. However, it can be seen that the $L^2$ error in Stokes does not initially reach the optimal order of 2, and with the gradual decrease in grid size $h$, the convergence order of the Stokes $L^2$ error also shows a decreasing trend. The reason for this decrease is influenced by the Darcy region, as the Darcy region also experiences a reduction in order at this point in time. Next, when observing the order of the velocity $H^1$, it can be noted that the $H^1$ order does not initially reach the optimal order, but with a decrease in $h$, the $H^1$ error order remains almost constant.
        
       The result of FEM-FEM error with varying $h$ but fixed $\epsilon=0.008$ are shown in Table \ref{tab:ex1fem}. 
       When observing errors in the Darcy part, we first fix the small parameter $\varepsilon$. As $h$ decreases and $h \gg \varepsilon$, we can observe that the convergence orders of the $L^2$ error and $H^1$ error for pressure become unstable. Comparing with Table \ref{tab:ex1msfem}, it is noticeable that for $h \gg \varepsilon$, the error accuracy of FEM-MsFEM is higher than that of FEM-FEM. For instance, when $h=1/16$, the accuracy of $e_{\phi_p}$ in the FEM-MsFEM algorithm can reach $10^{-5}$, whereas FEM-FEM can only achieve $10^{-4}$.
       When observing errors in the Stokes part, it can be noted that FEM-FEM does not reach the optimal order, and with the  decrease in grid size $h$, the convergence order of the Stokes $L^2$ error also shows a decreasing trend. The order of $H^1$ does not initially reach the optimal order, but with a decrease in $h$, the $H^1$ error order remains almost constant. This result is consistent with the situation in FEM-MsFEM and is influenced by the Darcy region.
       
       	FEM-MsFEM error at  $\frac{\epsilon}{h}=0.32$,  $N=32$. are shown in Table \ref{tab:ex1table3}. Our main goal is to observe errors in FEM-MsFEM when fixing $\varepsilon/h$, where $h_{finer}=h/N$. In the Darcy part, as $h$ gradually decreases, it can be observed that the $L^2$ and $H^1$ errors remain almost constant. This validates the importance of O($\varepsilon/h$) in the $L^2$ error and O($\sqrt{\varepsilon/h}$) in the $H^1$ error. In the Stokes region, the convergence orders of $L^2$ and $H^1$ errors fail to reach the optimal order, and there is no significant change in the orders.
        	\begin{table}[htb]
        	\caption{ FEM-MsFEM error with varying $h$ but fixed $\epsilon=0.008$ and $h_{finer}=1/2048$.}
        	\resizebox{\linewidth}{!}{
        		\begin{tabular}{c|r|r|r|r|r|r|r|r}
        			\hline\hline$h$ & $\left\|e_{\mathbf{u}_f}\right\|_{L^2}$ & Order & $\left\|e_{u_f}\right\|_{H^1}$ & Order & $\left\|e_{\phi_p}\right\|_{L^2}$ & Order & $\left\|e_{\phi_p}\right\|_{H^1}$ & Order \\
        			\hline 
        			1/4	&	6.07E-02	&		&	1.67E+00	&		&	8.99E-04	&		&	1.67E-02	&		\\
        			1/8	&	1.80E-02	&	1.76	&	9.44E-01	&	0.82	&	1.97E-04	&	2.19	&	7.56E-03	&	1.14	\\
        			1/16	&	4.79E-03	&	1.91	&	5.07E-01	&	0.90	&	8.43E-05	&	1.23	&	4.82E-03	&	0.65	\\
        			1/32	&	1.27E-03	&	1.92	&	2.69E-01	&	0.92	&	1.75E-04	&	-1.05	&	5.05E-03	&	-0.07	\\
        			1/64	&	3.43E-04	&	1.89	&	1.41E-01	&	0.93	&	3.25E-04	&	-0.89	&	6.82E-03	&	-0.44	\\
        			1/128	&	1.22E-04	&	1.49	&	7.41E-02	&	0.93	&	5.56E-04	&	-0.78	&	9.47E-03	&	-0.47	\\
        			1/256	&	5.79E-05	&	1.07	&	3.85E-02	&	0.95	&	3.79E-04	&	0.55	&	7.85E-03	&	0.27	\\
        			
        			\hline
        		\end{tabular}
        	\label{tab:ex1msfem}
        	}
        \end{table}

        \begin{table}
        	\caption{FEM-FEM error with varying $h$ but fixed $\epsilon=0.008$.}
        	\resizebox{\linewidth}{!}{
        		\begin{tabular}{c|r|r|r|r|r|r|r|r}
        			\hline\hline$h$ & $\left\|e_{\mathbf{u}_f}\right\|_{L^2}$ & Order & $\left\|e_{u_f}\right\|_{H^1}$ & Order & $\left\|e_{\phi_p}\right\|_{L^2}$ & Order & $\left\|e_{\phi_p}\right\|_{H^1}$ & Order \\
        			\hline
        			1/4	&	6.07E-02	&		&	 1.67E+00	&		&	5.31E-04	&		&	1.44E-02	&		\\
        			1/8	&	1.80E-02	&	1.76	&	9.44E-01	&	0.82	&	9.33E-04	&	-0.81	&	1.34E-02	&	0.10	\\
        			1/16	&	4.81E-03	&	1.90	&	5.07E-01	&	0.90	&	8.61E-04	&	0.12	&	1.26E-02	&	0.09	\\
        			1/32	&	1.28E-03	&	1.91	&	2.69E-01	&	0.92	&	9.03E-04	&	-0.07	&	1.27E-02	&	-0.01	\\
        			1/64	&	3.62E-04	&	1.83	&	1.41E-01	&	0.93	&	7.89E-04	&	0.19	&	1.19E-02	&	0.09	\\
        			1/128	&	1.55E-04	&	1.22	&	7.41E-02	&	0.93	&	8.76E-04	&	-0.15	&	1.23E-02	&	-0.04	\\
        			1/256	&	7.20E-05	&	1.11	&	3.85E-02	&	0.95	&	4.81E-04	&	0.86	&	8.73E-03	&	0.49	\\
        			\hline

        		\end{tabular}
        	\label{tab:ex1fem}
        	}
        \end{table}

        	\begin{table}
        		\caption{FEM-MsFEM error at  $\frac{\epsilon}{h}=0.32$,  $N=32$.}
        		\resizebox{\linewidth}{!}{
        			\begin{tabular}{c|r|r|r|r|r|r|r|r|r}
        			\hline	\hline$h$ & $\epsilon$ & $\left\|e_{u_p}\right\|_{L^2}$ & Order & $\left\|e_{u_f}\right\|_{H^1}$ & Order & $\left\|e_{\phi_p}\right\|_{L^2}$ & Order & $\left\|e_{\phi_p}\right\|_{H^1}$ & Order \\
        				\hline 1/16	&	0.02	&	4.80E-03	&		&	5.07E-01	&		&	3.05E-04	&		&	6.89E-03	&		\\
        				1/32	&	0.01	&	1.27E-03	&	1.92	&	2.69E-01	&	0.92	&	2.88E-04	&	0.08	&	5.99E-03	&	0.20	\\
        				1/64	&	0.005	&	3.38E-04	&	1.91	&	1.41E-01	&	0.93	&	2.80E-04	&	0.04	&	5.76E-03	&	0.06	\\
        				1/128	&	0.0025	&	9.71E-05	&	1.80	&	7.41E-02	&	0.93	&	2.71E-04	&	0.05	&	5.63E-03	&	0.03	\\
        				\hline
        			\end{tabular}
        		}
        		\label{tab:ex1table3}
        	\end{table}

         \subsubsection{Example 2}
         
          In this example, we consider nonseparable scales and solve (\ref{DARCY2}) with
         
       	$$
       \mathbb{K}_{\epsilon}=\frac{1}{4+P(\sin (2 \pi x / \epsilon)+\sin (2 \pi y / \epsilon))},
       $$
         \begin{figure}[htb]
         	\centering
         	\includegraphics[width=0.5\linewidth]{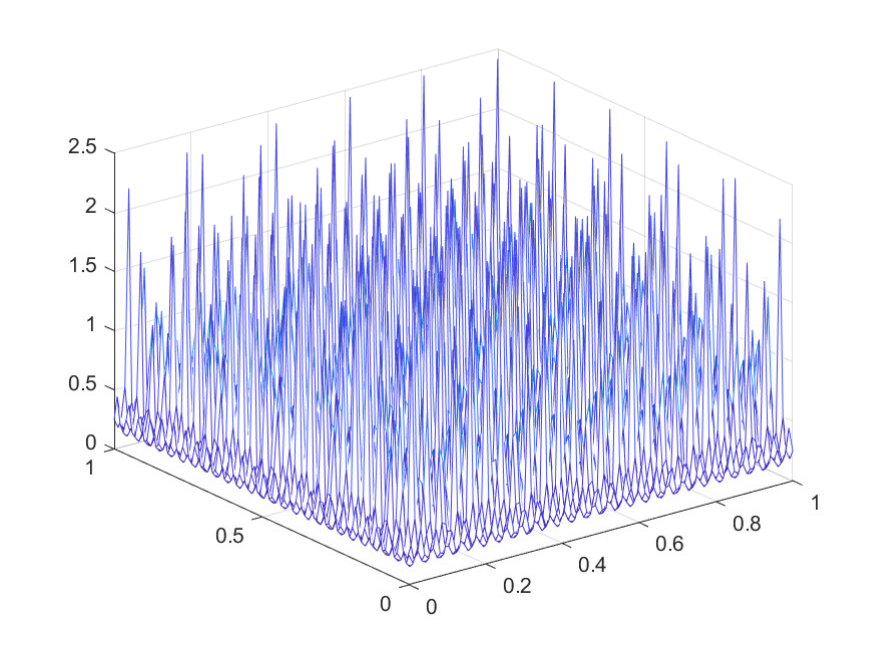}
         	\caption{$\mathbb{K}_{\epsilon}$ in Example 2}
         	\label{fig:k2}
         \end{figure}
         where $P$ is a parameter controlling the magnitude of the oscillation. We take $P=1.5$ in this example. The right hand side function $\mathbf{g}_f(x,y),g_p(x,y)$ is zero
         
         On $\partial \Omega_p$, we impose $\phi=0$ and on the $\partial \Omega_f$,	$u_f=[\sin (\pi x), 0],\text { on }(0,1) \times\{2\}$,
         $u_f=[0,0] , \text { on }\{0\} \times(1,2) \cup\{1\} \times(1,2)$.
         We set all physical parameters $\nu ,\alpha, g$ equal to 1. The Robin-Robin coefficients are set to $\gamma_{f}=0.1,\gamma_{p}=1$.


        The result of FEM-MsFEM error with varying $h$ but fixed $\epsilon=0.008$ and $h_{finer}=1/2048$ are shown in Table \ref{tab:ex2msfem}.  When examining errors in the Darcy part, we initially set a fixed small parameter and progressively refine the grid. As $h$ decreases, with $h \gg \varepsilon$, we observe a degradation in the convergence order of the $L^2$ error for $e_{\phi_p}$ from around $+2$ to approximately $-1$. This observation effectively confirms the presence of terms O($h^2$) to O($\varepsilon/h$) in the $L^2$ error analysis. Additionally, the convergence order of the $H^1$ error decreases from $+1$ to $-0.5$, providing robust validation for terms O($h$) to O($\sqrt{\varepsilon/h}$) in the $L^2$ error analysis.
         Moreover, when evaluating errors $e_{\mathbf{u}_f}$ in the Stokes part, standard MINI elements are employed in the Stokes region. However, it is evident that the $L^2$ error in Stokes does not initially attain the optimal order of 2. As the grid size $h$ gradually decreases, the convergence order of the Stokes $L^2$ error exhibits a declining trend. This decline is attributed to the influence of the Darcy region. Subsequently, when inspecting the order of the velocity $H^1$, it is observed that the $H^1$ order does not initially reach the optimal level, yet with a decrease in $h$, the $H^1$ error order remains relatively constant.

         The result of FEM-FEM error with varying $h$ but fixed $\epsilon=0.008$ in Table \ref{tab:ex2fem}. 
         When observing errors in the Darcy part, we first fix the small parameter $\varepsilon$. As $h$ decreases and $h \gg \varepsilon$, we can observe that the convergence orders of the $L^2$ error and $H^1$ error for pressure become unstable. Comparing with Table \ref{tab:ex2msfem}, it is noticeable that for $h \gg \varepsilon$, the error accuracy of FEM-MsFEM is higher than that of FEM-FEM. For instance, when $h=1/16$, the accuracy of $e_{\phi_p}$ in the FEM-MsFEM algorithm can reach $10^{-5}$, whereas FEM-FEM can only achieve $10^{-4}$.
         When observing errors in the Stokes part, it can be noted that FEM-FEM does not reach the optimal order, and with the  decrease in grid size $h$, the convergence order of the Stokes $L^2$ error also shows a decreasing trend. The order of $H^1$ does not initially reach the optimal order, but with a decrease in $h$, the $H^1$ error order remains almost constant. This result is consistent with the situation in FEM-MsFEM and is influenced by the Darcy region.
         
         FEM-MsFEM error at  $\frac{\epsilon}{h}=0.32$,  $N=32$. are shown in Table \ref{tab:ex2tab3}. Our main goal is to observe errors in FEM-MsFEM when fixing $\varepsilon/h$, where $h_{finer}=h/N$. In the Darcy part, as $h$ gradually decreases, it can be observed that the $L^2$ and $H^1$ errors remain almost constant. This validates the importance of O($\varepsilon/h$) in the $L^2$ error and O($\sqrt{\varepsilon/h}$) in the $H^1$ error. In the Stokes region, the convergence orders of $L^2$ and $H^1$ errors fail to reach the optimal order, and there is no significant change in the orders.

         \begin{table}[htb]
         	\caption{FEM-MsFEM error with varying $h$ but fixed $\epsilon=0.008$ and $h_{finer}=1/2048$.}
         	\resizebox{\linewidth}{!}{
         		\begin{tabular}{c|r|r|r|r|r|r|r|r}
         			\hline\hline$h$ & $\left\|e_{\mathbf{u}_f}\right\|_{L^2}$ & Order & $\left\|e_{u_f}\right\|_{H^1}$ & Order & $\left\|e_{\phi_p}\right\|_{L^2}$ & Order & $\left\|e_{\phi_p}\right\|_{H^1}$ & Order \\
         			\hline 
         			1/4	&	6.07E-02	&		&	1.67E+00	&		&	2.00E-03	&		&	3.12E-02	&		\\
         			1/8	&	1.80E-02	&	1.76	&	9.44E-01	&	0.82	&	5.06E-04	&	1.98	&	1.36E-02	&	1.20	\\
         			1/16	&	4.79E-03	&	1.91	&	5.07E-01	&	0.90	&	8.65E-05	&	2.55	&	6.93E-03	&	0.97	\\
         			1/32	&	1.26E-03	&	1.92	&	2.69E-01	&	0.92	&	6.31E-05	&	0.45	&	5.08E-03	&	0.45	\\
         			1/64	&	3.34E-04	&	1.92	&	1.41E-01	&	0.93	&	1.23E-04	&	-0.96	&	5.89E-03	&	-0.21	\\
         			1/128	&	9.44E-05	&	1.82	&	7.41E-02	&	0.93	&	2.28E-04	&	-0.90	&	7.94E-03	&	-0.43	\\
         			1/256	&	3.28E-05	&	1.53	&	3.85E-02	&	0.95	&	1.70E-04	&	0.43	&	6.75E-03	&	0.23	\\
         			
         			\hline
         		\end{tabular}
         	}
         	\label{tab:ex2msfem}
         \end{table}
         
         \begin{table}[htb!]
         	\caption{FEM-FEM error with varying $h$ but fixed $\epsilon=0.008$.}
         	\resizebox{\linewidth}{!}{
         		\begin{tabular}{c|r|r|r|r|r|r|r|r}
         		\hline	\hline$h$ & $\left\|e_{\mathbf{u}_f}\right\|_{L^2}$ & Order & $\left\|e_{u_f}\right\|_{H^1}$ & Order & $\left\|e_{\phi_p}\right\|_{L^2}$ & Order & $\left\|e_{\phi_p}\right\|_{H^1}$ & Order \\
         			\hline 
         			1/4	&	6.07E-02	&		&	1.67E+00	&		&	1.68E-03	&		&	2.91E-02	&		\\
         			1/8	&	1.80E-02	&	1.76	&	9.44E-01	&	0.82	&	4.96E-04	&	1.76	&	1.53E-02	&	0.93	\\
         			1/16	&	4.80E-03	&	1.91	&	5.07E-01	&	0.90	&	3.30E-04	&	0.59	&	1.17E-02	&	0.39	\\
         			1/32	&	1.27E-03	&	1.92	&	2.69E-01	&	0.92	&	4.21E-04	&	-0.35	&	1.07E-02	&	0.13	\\
         			1/64	&	3.37E-04	&	1.92	&	1.41E-01	&	0.93	&	2.68E-04	&	0.65	&	9.88E-03	&	0.11	\\
         			1/128	&	1.03E-04	&	1.70	&	7.41E-02	&	0.93	&	3.69E-04	&	-0.46	&	1.00E-02	&	-0.02	\\
         			1/256	&	3.73E-05	&	1.47	&	3.85E-02	&	0.95	&	2.11E-04	&	0.81	&	7.47E-03	&	0.42	\\
         			\hline
         		\end{tabular}
         	}
         \label{tab:ex2fem}
         \end{table}
         
         \begin{table}
         	\caption{FEM-MsFEM error at  $\frac{\epsilon}{h}=0.32$,  $N=32$.}
         	\resizebox{\linewidth}{!}{
         		\begin{tabular}{c|r|r|r|r|r|r|r|r|r}
         				\hline\hline$h$ & $\epsilon$ & $\left\|e_{u_p}\right\|_{L^2}$ & Order & $\left\|e_{u_f}\right\|_{H^1}$ & Order & $\left\|e_{\phi_p}\right\|_{L^2}$ & Order & $\left\|e_{\phi_p}\right\|_{H^1}$ & Order \\
         			\hline 1/16	&	0.02	&	4.79E-03	&		&	5.07E-01	&		&	1.16E-04	&		&	8.05E-03	&		\\
         			1/32	&	0.01	&	1.27E-03	&	1.92	&	2.69E-01	&	0.92	&	9.57E-05	&	0.28	&	5.76E-03	&	0.48	\\
         			1/64	&	0.005	&	3.33E-04	&	1.92	&	1.41E-01	&	0.93	&	9.42E-05	&	0.02	&	5.13E-03	&	0.17	\\
         			1/128	&	0.0025	&	8.87E-05	&	1.91	&	7.41E-02	&	0.93	&	9.07E-05	&	0.06	&	4.90E-03	&	0.07	\\
         			\hline
         		\end{tabular}
         	}
         \label{tab:ex2tab3}
         \end{table}
         
%
%
%


	\clearpage
	
	\bibliographystyle{plain}
	\bibliography{reference}
\end{document}